\documentclass[a4paper,12pt]{amsart}

\usepackage{amsfonts}
\usepackage{amsmath}
\usepackage{amssymb}
\usepackage{mathrsfs}
\usepackage[colorlinks]{hyperref}

\usepackage{graphicx}

\usepackage[usenames]{color}

\newtheorem{thm}{Theorem}[section]

\numberwithin{equation}{section}

\theoremstyle{definition}

\begin{document}

\title[Inverse problems for $q$-analogue of the heat equation]
{Inverse problems for $q$-analogue of the heat equation}

\author[Erkinjon Karimov]{Erkinjon Karimov}
\address{
  Erkinjon Karimov:
  \endgraf
  Fergana State University
  \endgraf
 19 Murabbiylar str., Fergana, 140100
  \endgraf
 Uzbekistan
 \endgraf
   and
  \endgraf
  V.I.Romanovskiy Institute of Mathematics
  \endgraf
 9 Universitet str., Tashkent, 100174
 \endgraf
 Uzbekistan
  \endgraf
 {\it E-mail address} {\rm erkinjon@gmail.com}
  }

 \author[Serikbol Shaimardan]{Serikbol Shaimardan}
\address{
  Serikbol Shaimardan:
  \endgraf
  L. N. Gumilyov Eurasian National University
  \endgraf
  5 Munaytpasov str., Astana, 010008
  \endgraf
  Kazakhstan
  \endgraf
  {\it E-mail address} {\rm shaimardan.serik@gmail.com}
  }
\thanks{SS was supported in parts by the MESRK (Ministry of Education and Science of the Republic of Kazakhstan) grant AP08052208.}

\date{}

\begin{abstract}
In this paper we explore the  weak solution of a time-dependent inverse source problem and inverse initial problem for $q$-analogue of the heat equation. As an  over-determination condition we have used integral type condition on space-variable (in the case of inverse source problem) and final time condition (in the case of inverse initial problem). Series form of considered problems have been obtained via the method of spectral expansions.

\end{abstract}

\subjclass[2010]{34C10, 39A10, 26D15.}

\keywords{Time-dependent inverse source problem; inverse initial problem; $q$-analogue of the heat equation; $q$-calculus; weak solution}

\maketitle

\section{Introduction}

In the last decade,  the theory of quantum groups and $q$-deformed algebras have been the subject of intense investigations. Many applications in physics have been investigated on the basis of the $q$-deformation of the Heisenberg algebra (see \cite{BM2000} and \cite{HW1999}). For instance, the $q$-deformed Schrodinger equations
have been proposed in  \cite{Micu1999}, \cite{Lavagno2009}  and applications to the study of $q$-deformed versions of the hydrogen atom, quantum harmonic oscillator  have been presented in \cite{EM2006}. Fractional calculus and $q$-deformed Lie algebras are closely related. For instance, in a sense of expansion the scope of standard Lie algebras to describe generalized symmetries. A new class of fractional $q$-deformed Lie algebras is proposed, which for the first time allows a smooth transition between different Lie algebras \cite{Herrmann2010}.

The origin of the  $q$-difference calculus can be traced back to the works \cite{J1908, J1910} by  F. Jackson and \cite{C1912} by R.D. Carmichael at the beginning of the twentieth century, while basic definitions and properties can be found e.g. in the monographs \cite{CK2000}-\cite{E2002}.  Recently,  the fractional $q$-difference calculus has been proposed by W. Al-Salam  \cite{A1966} and R.P. Agarwal \cite{A1969}. Today, maybe due to the explosion in research within the fractional differential calculus setting, new developments in this theory of fractional $q$-difference calculus have been addressed extensively by several researchers. For example, some researchers obtained $q$-analogues of the integral and differential fractional operators properties such as the $q$-Laplace transform and $q$-Taylor’s formula \cite{PMS2007}, $q$-Mittag-Leffler function \cite{A1966}. Moreover, in 2007, M.S. Ben Hammouda and Akram Nemri defined the higher-order $q$-Bessel translation and the higher-order $q$-Bessel Fourier transform, and they establishes some of their properties and studied  the higher-order $q$-Bessel heat equation \cite{HM2007}. In 2012, A. Fitouhi and F. Bouzeffour established in a great detail the $q$-Fourier analysis related to the $q$-cosine and constructed  the $q$-solution source, the $q$-heat polynomials and solve the $q$-analytic Cauchy problem. 

Regarding the inverse problems for different type partial differential equations we refer readers to \cite{isak} (inverse source problems) and \cite{lesn} (inverse initial problems).

In the present research we aim to study the time-dependent inverse source and inverse initial-value problems for $q$-analogue of the heat equation involving a self-adjoint operator in space-variable. 

The paper is organized as follows: The main results are presented and proved in Section \ref{S3}. In order to not disturb these presentations, we include first some necessary Preliminaries \ref{S2}.

\section{Preliminaries}\label{S2}

In this section, we recall key facts of $q$-calculus. We will always assume that $0<q<1$. 

The $q$-real number $[\alpha ]_q$ is defined by
$$
[\alpha ]_{q}:=\frac{1-q^{\alpha }}{1-q}.
$$

The $q$-derivative (or Jackson's  $q$-derivative) $D_{q}f(x)$ is defined as follows:
\begin{eqnarray}\label{additive2.1}
D_{q}f(x)=\frac{f(x)-f(qx)}{x(1-q)}.
\end{eqnarray}

The $q$-derivative of a product of two functions will have the following form: 
\begin{eqnarray}\label{additive2.2}
D_{q}(fg)(x)= f(qx)D_{q}(g)(x)+D_{q}(f)(x)g(x).
\end{eqnarray}

As given in \cite{CK2000}, two  $q$-analogue of the exponential function are defined as  
 \begin{eqnarray*}
 e_q^x=\sum\limits_{k=0}^\infty\frac{x^k}{[k]_q}, \;\;\;\;
  E_q^x=\sum\limits_{k=0}^\infty q^{k(k-1)/2}\frac{x^k}{[k]_q}.
  \end{eqnarray*}

Moreover, we have  
\begin{eqnarray}\label{additive2.3}
D_qe_q^x=e_q^x, \;\;\; D_qE_q^{-x}=E_q^{-qx}, \;\;\;    e_q^xE_q^{-x}=1.
\end{eqnarray}

The $q$-cosine and $q$-sine $q$-trigonometric functions are defined by ( see \cite{FB2012}):
\begin{eqnarray*}
\cos(\lambda{z};q^2)&=&\sum\limits_{k=0}^\infty\frac{(-1)^kq^{k^2}(\lambda{z})^{2k}}{[2k]_{q}!}, \\
\sin(\lambda{z};q^2)&=&\sum\limits_{k=0}^\infty\frac{(-1)^kq^{k(k+1)}(\lambda{z})^{2k+1}}{[2k+1]_{q}!},
\end{eqnarray*}
respectively. Here  the $q$-analogue of the binomial coefficients $[n]_{q}!$ are defined by
\begin{equation*}
[n]_{q}!=\left\{
\begin{array}{l}
{1,\mathrm{\;\;\;\;\;\;\;\;\;\;\;\;\;\;\;\;\;\;\;\;\;\;\;\;\;\;\;\;\;\;\;\;%
\;if\;n}=\mathrm{0,}} \\
{[1]_{q}\times [2]_{q}\times \cdots \times [n]_{q},\mathrm{%
\;if\;n}\in \mathrm{N.\;\;}}%
\end{array}%
\right.
\end{equation*}

The $q$-integral (or Jackson integral) is defined on $[a, b]$ by (see \cite{J1910})
\begin{eqnarray}\label{additive2.4}
\int\limits_a^b f(t)d_{q}t=(1-q) \sum\limits_{m=0}^\infty q^{m}\left[bf(bq^{m})-af(aq^{m})\right].
\end{eqnarray}

The  $q$-analogue of the formula of integration by parts will have the form of 
\begin{eqnarray}\label{additive2.5}
\int\limits_a^bf(x)D_qg(x)d_qx=\left[f(b)g(b)-f(a)g(a)\right] -\int\limits_a^bg(qx)D_qf(x)d_qx.    
\end{eqnarray}

Let $L^2_q\left[0, 1\right]$ be the space of all real-valued functions defined on $[0, 1]$ such that 
\begin{eqnarray*}
\|f\|_{L^2_q\left[0, 1\right]}:= \left(\int\limits_0^1 |f(x)|^2d_qx\right)^\frac{1}{2}<\infty.  
\end{eqnarray*}

The space  $L^2_q\left[0, 1\right]$ is a separable Hilbert space with the inner product:
\begin{eqnarray*}
\langle f,g\rangle:= \int\limits_0^1 f(x) g(x)d_qx, \;\;\;f,g\in  L^2_q\left[0, 1\right].
\end{eqnarray*}

Now,  we introduce the study held by M.H. Annaby and Z.S. Mansour on a basic $q$-Sturm–Liouville eigenvalue problem in a Hilbert space \cite[Chapter 3]{AM2005}. They investigated the following $q$-Sturm–Liouville equation:
\begin{eqnarray*} 
 -\frac{1}{q}D_{q^{-1}}D_qy(x)+v(x)y(x)=\lambda y(x),\;\;\;
 (0\leq{x}\leq1;\lambda\in\mathbb{C}),
\end{eqnarray*}
where $v(\cdot)$ is defined on $[0,1]$ and continuous at zero. Let  $C^2_{q,0}[0,1]$ be the  space   of
all functions $y(\cdot)$  such that  $y$, $D_qy$ are continuous at zero. In particular,  we get the following operator:
\begin{eqnarray}\label{additive2.6}
 \mathcal{L}:=\left\{
  \begin{array}{ll}
 -\frac{1}{q}D_{q^{-1}}D_qy(x)=\lambda y(x),   \\
y(0)=y(1)=0
  \end{array}
\right.
\end{eqnarray}
for $0\leq{x}\leq1 $ and $\lambda\in\mathbb{C}$. The  operator (\ref{additive2.6})   is a self-adjoint on  $C^2_{q,0}[0,1]\cap L^2_q\left[0, 1\right]$(see \cite[Theorem 3.4.]{AM2012}).   Moreover, the eigenvalues $\{\lambda_k\}_{k=1}^\infty$ are the zeros of $\sin(\sqrt{\lambda};q^2)$, where  
\begin{eqnarray}\label{additive2.10}
\lambda_k=(1-q)^{-2}q^{-2k+2\mu_k^{-1/2}}, \;\;\; k\geq 1 
\end{eqnarray}
and $\sum\limits_{k=1}^\infty\mu_k<\infty$, $0\leq\mu_k\leq1$. For sufficiently large $k$ and the corresponding set of eigenfunctions
$\{\frac{\sin(\sqrt{\lambda_k};q^2)}{\sqrt{\lambda_k}}\}_{k=1}^\infty$ form an orthogonal basis of $L_q^2(0, 1)$. Thus,  we can identify $f \in L_q^2(0, 1)$ via Fourier series:    
\begin{eqnarray*}
f(x):=\sum\limits_{k=0}^\infty\langle{f,\phi_k}\rangle \phi_k(x),
\end{eqnarray*}
where 
\begin{eqnarray}\label{additive2.1}
\phi_k(x)=\frac{\sin(\sqrt{\lambda_k}x;q^2)}{\sqrt{\lambda_k}}.
\end{eqnarray}

{\it The Sobolev space associated with $\mathcal{L}$ }: 
The space $C^\infty_{\mathcal{L}}[0,1]:=\bigcap\limits_{k=0}^{\infty}Dom\left(\mathcal{L}^k\right)$ is called the space of test functions for $\mathcal{L}$, where 
\begin{eqnarray*}
Dom\left(\mathcal{L}^k\right):=\left\{f\in L_q^2[0,1]: \mathcal{L}^kf\in Dom(\mathcal{L}), k=0,1,2,\dots  \right\}.  
\end{eqnarray*}

For $k\in\mathbb{N}_0:=\mathbb{N}\cup\{0\}$ and $g\in C^\infty_{\mathcal{L}}[0,1]$ we introduce the Fréchet topology of  $C^\infty_{\mathcal{L}}[0,1]$   by the family of norms:  
$$
\|g\|_{C^k_{\mathcal{L}}[0,1]}:=\max\limits_{i\leq{k}}\|\mathcal{L}^ig\|_{L_q^2[0,1]}.
$$

The space of $\mathcal{L}$-distributions $\mathcal{D}'_{\mathcal{L}}[0,1]:=L\left(C^\infty_{\mathcal{L}}[0,1], \mathbb{C}\right)$ is the space of all linear continuous functionals on $C^\infty_{\mathcal{L}}[0,1]$.

Thus, for $s\in\mathbb{R}$ we can also define Sobolev spaces $W_{q,\mathcal{L}}^s$ associated to $\mathcal{L}$ in the following form:
\begin{eqnarray*}
W_{q,\mathcal{L}}^s:=\left\{f\in\mathcal{D}'_{\mathcal{L}}[0,1]: \mathcal{L}^{s/2}\in L_q^2[0,1]\right\},
\end{eqnarray*}
with the norm $\|f\|_{W_{q,\mathcal{L}}^s}:=\|\mathcal{L}^{s/2}f\|_{L_q^2[0,1]}$.

For $k\in\mathbb{N}_0$ we introduce the space  $C^k_{q,0}\left( [0, 1];W_{q,\mathcal{L}}^s[0,1]\right)$  defined by  the norms
\begin{eqnarray*}
\|u\|_{C_{q}\left( [0, 1]; W_{q,\mathcal{L}}^s[0,1]\right)}:= \sum\limits_{n=0}^k
 \max\limits_{0\leq t\leq 1}\| D_{q,t}^nu(t,.)\|_{W_{q,\mathcal{L}}^s[0,1]}, \;\;\;s\in\mathbb{N}_0, 
\end{eqnarray*} 
where the $q$-partial differential operator $D_{q,x}u(t,x)$ is given in the following form:
\begin{eqnarray*}
D_{q,x}u(t,x)=\frac{u(t,x)-u(t,qx)}{(1-q)x}.
\end{eqnarray*} 

Now we are ready to formulate the inverse  problems.

 \section{ Inverse problems}  \label{S3}

\subsection{Inverse source problem}

In this sub-section, we will seek a pair of functions $\{u(t,x), \upsilon(t)\}$,  which satisfies the the following $q$-analogue of the heat equation 
\begin{eqnarray}\label{additive3.1}
D_{q,t}u(t,x)+\mathcal{L}u(t,x)=\upsilon(t) f(t,x), \;\;\;0<x<1,\;\;t>0
\end{eqnarray} 
together with initial condition  
\begin{eqnarray}\label{additive3.2}
 u(0,x)=\varphi(x), \;\;\;0<x<1,\;\;t>0
\end{eqnarray}
and the over-determination condition
\begin{eqnarray}\label{additive3.3}
\int\limits_0^1u(t,x)d_qx=\psi(t), \;\;\; t>0,  
\end{eqnarray}
where $\varphi(x)$, $f(t,x)$ are given functions and $\psi(t)$ is the mass or total energy of the system, depending on the certain problem in physics, which is also given.

\begin{thm}\label{thm3.1}
Let  $\varphi\in W_{q,\mathcal{L}}^2[0,1]$,  $f\in C_q\left([0,T];W_{q,\mathcal{L}}^2[0,1]\right)$, $\phi\in C^1_q[0,T]$ and
\begin{eqnarray}\label{additive3.4}
\left|\int\limits_0^1f(t,x)d_qx\right|^{-1}\leq M_1, \;\;\; 0<t\leq T, 
\end{eqnarray}
where $M_1$ is positive real number. Then, the inverse problem is locally well-posed in time.
\end{thm}

\begin{proof} {\it Existence.}  We can write the solution of  problem (\ref{additive3.1})-(\ref{additive3.3})  in the series form
\begin{eqnarray}\label{additive3.5} 
u(x,t)=\sum\limits_{k=0}^\infty u_k(t)\phi_k(x),
\end{eqnarray}
where
\begin{eqnarray}\label{additive3.6} 
u_k(t)&=&e_q^{-t\lambda_k}\langle\varphi,\phi_k\rangle+e_q^{-t\lambda_k}\int\limits_0^tE_q^{qs\lambda_k}\upsilon(s)\langle{f}(s,\cdot),\phi_k\rangle d_qs.
\end{eqnarray}

Applying the operator $\mathcal{L}$ to (\ref{additive3.6}), we have
\begin{eqnarray}\label{additive3.7} 
\mathcal{L}u(x,t)&=&\sum\limits_{k=0}^\infty u_k(t)\mathcal{L}\phi_k(x) 
=\sum\limits_{k=0}^\infty u_k(t) \lambda_k \phi_k(x).
\end{eqnarray}

Now, by using  (\ref{additive3.5})  and  (\ref{additive3.1})  in (\ref{additive3.3}), one can obtain the result as: 
\begin{eqnarray*} 
D_{q}\psi(t)&=& \int\limits_0^1D_{q,t}u(t,x)d_qx \\
&=&\int\limits_0^1\left[\mathcal{L}u(t,x)+\upsilon(t) f(t,x)\right]d_qx\\
&=&\sum\limits_{k=0}^\infty u_k(t) \lambda_k \int\limits_0^1\phi_k(x)d_qx+\upsilon(t)\int\limits_0^1 f(t,x)d_qx. 
\end{eqnarray*}

Therefore, assuming that $\left|\int\limits_0^1 f(t,x)d_qx\right|^{-1}\le M_1$, we deduce
\begin{eqnarray}\label{additive3.8} 
\upsilon(t)=\frac{D_{q}\psi(t)}{ \int\limits_0^1 f(t,x)d_qx}-\frac{\sum\limits_{k=0}^\infty u_k(t) \lambda_k \int\limits_0^1\phi_k(x)d_qx}{ \int\limits_0^1 f(t,x)d_qx}. 
\end{eqnarray}

Now, combining (\ref{additive3.5}) and (\ref{additive3.6}) with (\ref{additive3.8}), we conclude that
\begin{eqnarray}\label{additive3.9} 
\upsilon(t)&=&\widehat{\psi}(t)-\int\limits_0^t\upsilon(s)K(t,s)d_qs,
\end{eqnarray}
where
\begin{eqnarray}\label{additive3.10} 
\widehat{\psi}(t)&=&\frac{D_{q,t}\psi(t)-\sum\limits_{k=0}^\infty e_q^{-t\lambda_k}\langle\varphi,\phi_k\rangle\lambda_k \int\limits_0^1\phi_k(x)d_qx}{ \int\limits_0^1 f(t,x)d_qx}
\end{eqnarray}
and 
\begin{eqnarray}\label{additive3.11} 
K(t,s)&=&\frac{ \sum\limits_{k=0}^\infty  e_q^{-t\lambda_k}\lambda_k \int\limits_0^1\phi_k(x)d_qx\int\limits_0^tE_q^{qs\lambda_k} \langle{f}(s,\cdot),\phi_k\rangle d_qs}{ \int\limits_0^1 f(t,x)d_qx} 
\end{eqnarray}

For $\psi\in C^1_q[0,T]$ we get $M_2:=\max\limits_{0\leq{t}\leq{T}}\left|D_{q,t}\psi(t)\right|$ and $\varphi\in W_{q,\mathcal{L}}^2[0,1]$. Then by using the Cauchy–Schwarz inequality and  Hoelder inequality, and also (\ref{additive2.3}), (\ref{additive3.4}), (\ref{additive3.10}), (\ref{additive3.11}),    we obtain that 
\begin{eqnarray*}
|\widehat{\psi}|&\leq& M_1\left[\left|D_{q,t}\psi(t)\right|+\sum\limits_{k=0}^\infty  \left|\langle\varphi,\lambda_k \phi_k\rangle\right|\int\limits_0^1\phi_k(x)d_qx\right]\\
&\leq& M_1\left[M_2+\sum\limits_{k=0}^\infty  \left|\langle\varphi,\mathcal{L} \phi_k\rangle\right| \|\phi_k\|_{L^2_q[0,1]} \right]\\
&\leq& M_1\left[M_2+\sum\limits_{k=0}^\infty  \left|\langle\mathcal{L} \varphi,\phi_k\rangle\right|^2 \|\phi_k\|^2_{L^2_q[0,1]} \right]\\
&\leq& M_1\left[M_2+\|\mathcal{L} \varphi\|^2_{L^2_q[0,1]} \sum\limits_{k=0}^\infty   \|\phi_k\|^2_{L^2_q[0,1]} \right]\\
&\leq& M_1\left(M_2+\| \varphi\|^2_{W_{q,\mathcal{L}}^2[0,1]}\right)<\infty
\end{eqnarray*}
and
\begin{eqnarray*} 
\left|K(t,s)\right|&\leq& M_1\left|\sum\limits_{k=0}^\infty  e_q^{-t\lambda_k} \int\limits_0^1\phi_k(x)d_qx\int\limits_0^tE_q^{qs\lambda_k}  \langle{f}(s,\cdot),\lambda_k\phi_k\rangle d_qs\right|\nonumber\\
&\leq& M_1\sum\limits_{k=0}^\infty   \left| \int\limits_0^1\phi_k(x)d_qx\right|\int\limits_0^t  \left|\langle{f}(s,\cdot),\mathcal{L}\phi_k\rangle \right|d_qs\nonumber\\
&\leq& M_1\sum\limits_{k=0}^\infty \|\phi_k(x)\|_{L^2_q[1,0]}\int\limits_0^t  \langle\mathcal{L}{f}(s,\cdot),\phi_k\rangle d_qs\nonumber\\
&\leq& TM_1\max\limits_{0\leq{t}\leq{T}}  \|\mathcal{L}{f}(s,\cdot)\|^2_{L^2_q[0,1]}\sum\limits_{k=0}^\infty \|\phi_k(x)\|^2_{L^2_q[1,0]}\nonumber\\
&\leq& TM_1  \| {f}(s,\cdot)\|^2_{C_q\left([0,T];W_{q,\mathcal{L}}^2[0,1]\right)}<\infty. 
\end{eqnarray*}
which means that $K$ is continuous on $C\left([0,T]\times[0,1]\right)$ and $\widehat{\psi}\in C[0,T]$.  Then the $q$-integral equations(\ref{additive3.9})  has a unique solution in $C[0,T]$ \cite[Theorem 3.4.]{{TC2014}} and it can be written via resolvent-kernel as follows
\[
\upsilon(t)=\widehat{\psi}(t)-\int\limits_0^t\widehat{\psi}(s)R(t,s)d_qs,
\]
where $R(t,s)$ is the resolvent-kernel of the $K(t,s)$.

{\it Convergence.} Next, we will derive the uniform convergence of $\mathcal{L}u(t, x)$ and $D_{q,t}u(t,x)$. Since $\|v\|_{C_q[0, T]}:=M_3<\infty$ and by using   and the Cauchy–Schwarz inequality and Hoelder inequality to (\ref{additive3.11}), we get that
\begin{eqnarray*} 
\left|u_k(t)\right|&\leq&  \left|\langle\varphi,\phi_k\rangle\right| 
+\int\limits_0^t\left|v(t)\right|\left|\langle f(s,\cdot),\phi_k\rangle\right|d_qs\nonumber\\
&\leq&\left|\langle\varphi,\phi_k\rangle\right| 
+\|v\|_{C_q[0,T]}\int\limits_0^t\left|\langle f(s,\cdot),\phi_k\rangle\right|d_qs\nonumber\\
&\leq&\left[\|\varphi\|^2_{L_q^2[0,1]} +      T\max\limits_{0\leq{t}\leq{T}}\|v\|_{C_q[0,T]}\|f(s,\cdot)\|^2_{L_q^2[0,1]}\right] 
\|\phi_k\|^2_{L_q^2[0,1]}
\end{eqnarray*}
and 
\begin{eqnarray*}  
\left|D_qu_k(t)\right|&\leq&\lambda_ku_k(t)+|v(t)f_k(t)|\nonumber\\
&\leq&\left[\|\mathcal{L}\varphi\|^2_{L_q^2[0,1]} +T\max\limits_{0\leq{t}\leq{T}}\|\mathcal{L}f(t,\cdot)\|^2_{L_q^2[0,1]}\right.\nonumber\\ &+&\left.\|v\|_{C_q[0,T]}\max\limits_{0\leq{t}\leq{T}}\|f(t,\cdot)\|^2_{L_q^2[0,1]}\right]\|\phi_k\|^2_{L_q^2[0,1]}.  \end{eqnarray*} 
Similarly, one can get 
\begin{eqnarray*} 
 \left|\mathcal{L}u(t,\cdot)\right| &=&\sum\limits_{k=1}^\infty \left|u_k(t)\right| \lambda_k \phi_k(x)\nonumber\\
 &\leq&\left[\|\mathcal{L}\varphi\|^2_{L_q^2[0,1]} +T\|v\|_{C_q[0,T]}\max\limits_{0\leq{t}\leq{T}}\|\mathcal{L}f(s,\cdot)\|^2_{L_q^2[0,1]}  \right]\|\phi_k\|^2_{L_q^2[0,1]} .
\end{eqnarray*}

Thus, using (\ref{additive3.5})-(\ref{additive3.7})  we have that
\begin{eqnarray*}
\| u(t, x)\|_{L_q^2[0,1]}&\leq&\|\varphi\|_{L_q^2[0,1]} +\|v\|_{C_q[0,T]}\|f\|_{C_q\left([0,T];W_{q,\mathcal{L}}^2[0,1]\right)}<\infty,
\end{eqnarray*}
\begin{eqnarray*}
\| D_{q,t}u(t, x)\|_{L_q^2[0,1]}&\leq&\| \varphi\|_{W_{q,\mathcal{L}}^2[0,1]} +\|v\|_{C_q[0,T]} \| f\|_{C_q\left([0,T];W_{q,\mathcal{L}}^2[0,1]\right)}<\infty 
\end{eqnarray*}
and 
\begin{eqnarray*}
\|\mathcal{L}u(t, x)\|_{L_q^2[0,1]}&\leq&\|\varphi\|_{W_{q,\mathcal{L}}^2[0,1]} + \|v\|_{C_q[0,T]}\| f \|_{C_q\left([0,T];W_{q,\mathcal{L}}^2[0,1]\right)}<\infty,
\end{eqnarray*}
which means that $u\in C_q\left([0,1];L_q^2[0,1]\right) \cap C_q\left([0,T];W_{q,\mathcal{L}}^2[0,1]\right)$.
The uniqueness of the solution for the proposed problem can be obtained based on above-given estimates, following standard procedure.

\end{proof}

\subsection{Inverse initial problem.}
In the present sub-section we are interested to find a pair of functions $u(t,x), \tau(x)$, which satisfies (\ref{additive3.1}) together with non-local condition in time
\begin{equation}\label{additive3.12}
    u(T,x)=\alpha u(0,x)+\tau(x)
\end{equation}
and the over-determination condition
\begin{equation}\label{additive3.13}
u(\xi_0,x)=\nu(x),\,\,0<x<1,    
\end{equation}
where $\xi_0\in(0,T]$, $|\alpha|\le 1$ are given real numbers, $\nu(x)$ is a given function. Moreover, we assume that function $\upsilon(t)f(t,x)$ is also given function.

The following statement is true:
\begin{thm}\label{thm3.2}
Let  $\tau\in W_{q,\mathcal{L}}^2[0,1]$, $\upsilon(t)\in C_q[0,T]$, $f\in C_q\left([0,T];W_{q,\mathcal{L}}^2[0,1]\right)$, $\nu\in C^1_q[0,T]$  
Then, the inverse initial problem is locally well-posed in time.
\end{thm}
\begin{proof}
We will highlight key moments of the proof of this statement. 

Looking for the solution of this problem as (\ref{additive3.5}) and temporarily using notation $u(0,x)=\gamma(x)$, we get for $u_k(t)$ expression given by (\ref{additive3.6}). Non-local condition (\ref{additive3.12}) gives us
\begin{equation}\label{additive3.14}
\left(e_q^{-T\lambda_k}-\alpha\right)\langle\gamma,\phi_k\rangle+e_q^{-T\lambda_k}\int\limits_0^TE_q^{qs\lambda_k}\upsilon(s)\langle f(s,\cdot),\phi_k\rangle d_qs=\langle\tau,\phi_k\rangle.    
\end{equation}
In order to find $\gamma(x)$ and $\tau(x)$, we will use over-determination condition (\ref{additive3.13}):
\begin{equation}\label{additive3.15}
\gamma(x)=e_q^{\xi_0\lambda_k}\nu(x)-\int\limits_0^{\xi_0}E_q^{qs\lambda_k}\upsilon(s)f(s,x)d_qs.   
\end{equation}
 Considering (\ref{additive3.14}) and (\ref{additive3.15}), one can easily find that 
 \[
 \tau(x)=e_q^{\xi_0\lambda_k}\left(e_q^{-T\lambda_k}-\alpha\right)\nu(x)+e_q^{-T\lambda_k}\int\limits_0^TE_q^{qs\lambda_k}\upsilon(s)f(s,x)d_qs+
 \]
 \[
 +\left(\alpha e_q^{-\xi_0\lambda_k}-e_q^{-(\xi_0+T)\lambda_k}\right)\int\limits_0^{\xi_0}E_q^{qs\lambda_k}\upsilon(s)f(s,x)d_qs
 \]
Remained part of the proof will be done similarly as in the previous Theorem.
\end{proof}

\section{Conflict of Interests}

Authors declare that they have no conflict of interests

 \end{document}